\documentclass[a4paper, 12pt, reqno]{amsart}
\usepackage{amsfonts}
\usepackage{latexsym}

\usepackage{amsmath, amssymb, mathrsfs, amsthm, cases, verbatim}
\usepackage{fancybox,color}
\usepackage{enumerate}
\usepackage[active]{srcltx}
\usepackage[colorlinks]{hyperref}
\usepackage{hypernat}
\usepackage{graphicx}
\usepackage{tikz}

\usepackage[left=3.5cm,right=3.5cm, top=3cm, bottom=3cm]{geometry}

\theoremstyle{plain}
\newtheorem{theorem}{Theorem}[section]
\newtheorem{lemma}[theorem]{Lemma}

\theoremstyle{definition}
\newtheorem{definition}[theorem]{Definition} \theoremstyle{remark}

\begin{document}
\newcommand{\RR}{{\mathbb R}}
 \newcommand{\N}{{\mathbb N}}
\newcommand{\Rn}{{\RR^n}}
\newcommand{\ieq}{\begin{equation}}
\newcommand{\eeq}{\end{equation}}
\newcommand{\ieqa}{\begin{eqnarray}}
\newcommand{\eeqa}{\end{eqnarray}}
\newcommand{\ieqas}{\begin{eqnarray*}}
\newcommand{\eeqas}{\end{eqnarray*}}
\newcommand{\Bo}{\put(260,0){\rule{2mm}{2mm}}\\}

\numberwithin{equation}{section}

\def\at#1{{\bf #1}: }
\def\att#1#2{{\bf #1}, {\bf #2}: }
\def\attt#1#2#3{{\bf #1}, {\bf #2}, {\bf #3}: }
\def\atttt#1#2#3#4{{\bf #1}, {\bf #2}, {\bf #3},{\bf #4}: }
\def\aug#1#2{\frac{\displaystyle #1}{\displaystyle #2}}
\def\figura#1#2{
\begin{figure}[ht]
\vspace{#1}
\caption{#2}
\end{figure}}
\def\B#1{\bibitem{#1}}
\def\q{\int_{\Omega^\sharp}}
\def\z{\int_{B_{\bar{\rho}}}\underline{\nu}\nabla (w+K_{c})\cdot
\nabla h}
\def\a{\int_{B_{\bar{\rho}}}}
\def\b{\cdot\aug{x}{\|x\|}}
\def\n{\underline{\nu}}
\def\d{\int_{B_{r}}}
\def\e{\int_{B_{\rho_{j}}}}
\def\LL{{\mathcal L}}
\def\D{{\mathcal D}}
\def\itr{\mathrm{Int}\,}
\def\tg{\tilde{g}}
\def\A{{\mathcal A}}
\def\S{{\mathcal S}}
\def\H{{\mathcal H}}
\def\M{{\mathcal M}}
\def\MM{{\mathcal M}}
\def\T{{\mathcal T}}
\def\N{{\mathcal N}}
\def\I{{\mathcal I}}
\def\F{{\mathcal F}}
\def\J{{\mathcal J}}
\def\E{{\mathcal E}}
\def\P{{\mathcal P}}
\def\HH{{\mathcal H}}
\def\V{{\mathcal V}}
\def\B{{\mathcal B}}
\def\BB{{\mathbb B}}

\def\vint_#1{\mathchoice%
          {\mathop{\kern 0.2em\vrule width 0.6em height 0.69678ex depth -0.58065ex
                  \kern -0.8em \intop}\nolimits_{\kern -0.4em#1}}%
          {\mathop{\kern 0.1em\vrule width 0.5em height 0.69678ex depth -0.60387ex
                  \kern -0.6em \intop}\nolimits_{#1}}%
          {\mathop{\kern 0.1em\vrule width 0.5em height 0.69678ex
              depth -0.60387ex
                  \kern -0.6em \intop}\nolimits_{#1}}%
          {\mathop{\kern 0.1em\vrule width 0.5em height 0.69678ex depth -0.60387ex
                  \kern -0.6em \intop}\nolimits_{#1}}}
\def\vintslides_#1{\mathchoice%
          {\mathop{\kern 0.1em\vrule width 0.5em height 0.697ex depth -0.581ex
                  \kern -0.6em \intop}\nolimits_{\kern -0.4em#1}}%
          {\mathop{\kern 0.1em\vrule width 0.3em height 0.697ex depth -0.604ex
                  \kern -0.4em \intop}\nolimits_{#1}}%
          {\mathop{\kern 0.1em\vrule width 0.3em height 0.697ex depth -0.604ex
                  \kern -0.4em \intop}\nolimits_{#1}}%
          {\mathop{\kern 0.1em\vrule width 0.3em height 0.697ex depth -0.604ex
                  \kern -0.4em \intop}\nolimits_{#1}}}

\newcommand{\kint}{\vint}
\newcommand{\intav}{\vint}
\newcommand{\kintint}[2]{\mathchoice%
          {\mathop{\kern 0.2em\vrule width 0.6em height 0.69678ex depth -0.58065ex
                  \kern -0.8em \intop}\nolimits_{\kern -0.45em#1}^{#2}}%
          {\mathop{\kern 0.1em\vrule width 0.5em height 0.69678ex depth -0.60387ex
                  \kern -0.6em \intop}\nolimits_{#1}^{#2}}%
          {\mathop{\kern 0.1em\vrule width 0.5em height 0.69678ex depth -0.60387ex
                  \kern -0.6em \intop}\nolimits_{#1}^{#2}}%
          {\mathop{\kern 0.1em\vrule width 0.5em height 0.69678ex depth -0.60387ex
                  \kern -0.6em \intop}\nolimits_{#1}^{#2}}}

\renewcommand{\div}{\operatorname{div}}
\newcommand{\trace}{\operatorname{Trace}}

\title[Characterization of $p$-Harmonic Functions]{On the Characterization of $p$-Harmonic Functions  on the Heisenberg Group by Mean Value Properties.}

\author[F. Ferrari]
{Fausto Ferrari}
\address{Dipartimento di Matematica dell'Universit\`a di Bologna, Piazza di Porta S. Donato, 5, 40126, Bologna, Italy}
\email{fausto.ferrari@unibo.it}
\author[Q. Liu]
{Qing Liu}
\address{Department of Mathematics,
University of Pittsburgh,
Pittsburgh, PA 15260, USA}
\email{qingliu@pitt.edu}
\author[J. Manfredi]
{Juan J. Manfredi}
\address{Department of Mathematics,
University of Pittsburgh,
Pittsburgh, PA 15260, USA}
\email{manfredi@pitt.edu}

\date{\today}

\thanks{F.~F.\ is supported by  MURST, Italy, by University of Bologna,Italy
by EC project CG-DICE and by the ERC starting grant project 2011 EPSILON (Elliptic PDEs and Symmetry of Interfaces and Layers for Odd Nonlinearities). F.~F.\ wishes to thank the Department of Mathematics at the University of Pittsburgh for the kind hospitality.\\ \indent
Q.~L, and J.~M. supported by NSF award DMS-1001179}



\begin{abstract}
We  characterize  $p-$harmonic functions in the Heisenberg group in terms of an asymptotic mean value property, where $1<p<\infty$,   following the scheme
described  in \cite{MPR} for  the Euclidean case. The new tool that allows us to consider the subelliptic case is a geometric lemma, Lemma \ref{lemma1} below,  that relates the directions of the points of maxima and minima of a function on a small subelliptic ball with the unit horizontal gradient of that function. 
\end{abstract}

\maketitle
\section{Introduction}
In this paper we study $p- $harmonic functions in the Heisenberg group  in terms of an asymptotic mean value property.  The  corresponding characterization of the $p$-harmonic functions in terms of an asymptotic mean value property in the Euclidean sense was obtained in \cite{MPR}. More precisely  in \cite{MPR},  the authors show
 that if $u$ is a continuous function in a domain $\Omega\subset \mathbb{R}^n$ and $p\in (1 , \infty]$,  then the asymptotic expansion
$$
u(x)=\frac{p-2}{2(p+n)}\left\{\max_{\overline{B_\epsilon(x)}} u+\min_{\overline{B_\epsilon(x)}} u\right\}+\frac{2+n}{p+n}\kint_{B_\epsilon(x)}u(y)\,dy+o(\epsilon^2),
$$   
holds as $\epsilon\to 0$  for all $x\in \Omega$ in the viscosity sense if and only if $u$ is a viscosity solution of the $p$-Laplace equation
$$
\div\left(|\nabla u(x)|^{p-2}\nabla u(x)\right)=0.
$$
 Here $B_\epsilon(x)$ is the Euclidean ball centered in $x$ with radius $\epsilon$.

We want to extend this  characterization to functions defined on the Heisenberg group $\mathbb{H}^n$. In Section \ref{preliminary} we present an overview of the Heisenberg group, where the geometry and analysis are  different than in  Euclidean space. For this introduction  we anticipate a few definitions  from 
Section \ref{preliminary} and refer the reader to this section for full details.  For  $p\in (1,\infty)$ the (subelliptic) $p$-Laplace  operator in the Heisenberg group is 
$$
\Delta_{p, \mathbb{H}^n}u=\div_{\mathbb{H}^n}\left(| \nabla_{\mathbb{H}^n}u|^{p-2} \nabla_{\mathbb{H}^n}u\right).
$$

 Here $\nabla_{\mathbb{H}^n}$ and  $\div_{\mathbb{H}^n}$  are respectively the intrinsic gradient and  the intrinsic divergence in the Heisenberg group.  From now on we will  denote by  ${B}(P,\epsilon)$  the intrinsic ball of radius $\epsilon$  with respect to the gauge distance, centered at the point $P\in \mathbb{H}^n$.

 We point out that for $p=2,$ $\Delta_{2, \mathbb{H}^n}u=\Delta_{\mathbb{H}^n}u$ is the real part of the Kohn-Laplace operator, a linear degenerate  second order elliptic operator whose lowest eigenvalue  is always zero. In particular,  to give an idea of the structure of this linear operator corresponding to the case $p=2$ and  $n=1$  we set
$$
\Delta_{\mathbb{H}}u(P)=\mbox{div}_{\mathbb{H}}(M(P)\nabla u(P))=\trace(M(P)D^2u(P)),
$$
where the $3\times 3$ matrix $M(P)$ for $P=(x, y ,t)\in \mathbb{H}$ is given by 
$$
M(P)=\left[\begin{array}{lcc}1,&0,&2 y\\
0,&1,&-2x \\
2y ,&-2x,&4(x^2+y^2)\end{array}\right].
$$
Observe that we always have
$$
\min\bigg\{\lambda\colon\lambda\mbox{ is an eigenvalue of } M(P)\bigg\}=0.
$$
 
For $p=1$ we get a subelliptic version of the mean curvature operator
$$
\Delta_{1, \mathbb{H}^n}u(P)=\div_{\mathbb{H}^n}\left(\frac{ \nabla_{\mathbb{H}^n}u(P)}{|\nabla_{\mathbb{H}^n}u(P)|}\right).
$$

See  the monograph  \cite{CDPT}  for the intrinsic mean curvature operator  in the Heisenberg setting.  See also  \cite{CC} and  \cite{FLM} for some recent results on the  the flow by mean curvature in the subelliptic setting.

Our main result is the following:

\begin{theorem}\label{maintheorem}
Let $1<p<\infty$ and let $u$ be a continuous function defined  in a domain $\Omega\subset\mathbb{H}^n.$ The asymptotic expansion
\begin{equation}\label{asymptotic}
u(P)=\frac{\alpha}{2}\left(\min_{\overline{B(P,\epsilon)}}u +\max_{B(P,\epsilon)}u\right)+  \beta \kint_{B(P,\epsilon)}u(x,y,t)+o(\epsilon^2),
\end{equation}
holds  as $\epsilon\to 0$ for every $P\in \Omega$ in the viscosity sense if and only if
$$
\Delta_{p,  \mathbb{H}^n}u=0
$$
in $\Omega$ in the viscosity sense,
where
 $$\alpha=\frac{2(p-2)C(n)}{2(p-2)C(n)+1},\,  \beta=\frac{1}{2(p-2)C(n)+1}, \textrm{ and }
C(n)=\frac{1}{2(n+1)}\frac{\int_{0}^{1}(1-s^2)^{\frac{n+1}{2}}\,dt}{\int_{0}^{1}(1-s^2)^{\frac{n}{2}}\,dt}.
$$ 
\end{theorem}
We remark that $\alpha+\beta=1$ and that 
$$C(n)= \frac{\Gamma \left(\frac{n+3}{2}\right)^2}{(2
   n+1) \Gamma \left(\frac{n}{2}+1\right)
   \Gamma \left(\frac{n}{2}+2\right)}.$$

The key tool in the proof of Theorem \ref{maintheorem} is Lemma \ref{lemma1} in Section \ref{Mean_formulas}. Roughly speaking,  in Lemma \ref{lemma1} we prove that  if $P_{0}=(x_{0},y_{0},t_{0})\in\mathbb{H}^n$ is not characteristic for the level set $\{u=u(P_{0})\},$ then the extrema  of the function $u$ in the intrinsic ball of radius $\epsilon$ and center $P_{0}$ are attained at points $(x_\epsilon, y_\epsilon, t_\epsilon)$ that satisfy
$$
\lim_{\epsilon\to 0}\left(
\frac{x_\epsilon-x_{0}}{\epsilon},\frac{y_\epsilon-y_{0}}{\epsilon},\frac{|t_\epsilon-t_{0}|}{\epsilon^3}
\right)=\left(\pm \left(\frac{ \nabla_{\mathbb{H}^n}u(P_{0})}{|\nabla_{\mathbb{H}^n}u(P_{0})|}\right),\frac{1}{2}|\mathfrak{p}(P_{0})|\right),
$$ where 
 $ \mathfrak{p}$ is the imaginary curvature
of the level surface $\{u=u(P_{0})\}$ at $P_{0}$  introduced in \cite{AF1} and \cite{AF2} and given by
$$
\mathfrak{p}(P_{0})=- \frac{ [X,Y](P_{0})}{|\nabla_{\mathbb{H}^n}u(P_{0})|}.
$$

We recall now the definition of viscosity solution in the Heisenberg group taken from \cite{Bi}. 

Observe that  if $u$ is smooth then
$$
-\Delta_{p, \mathbb{H}^n}u=-|\nabla_{\mathbb{H}^n}u|^{p-2}\left((p-2)\Delta_{\infty, \mathbb{H}^n}u+\Delta_{\mathbb{H}^n}u\right), 
$$
where we have used 
\begin{equation}\label{infinitylaplacian}
\begin{split}
\Delta_{\infty,\mathbb{H}^n}u= \langle D_{\mathbb{H}^n}^{2,*}u \,\,\frac{\nabla_{\mathbb{H}^n}u}{|\nabla_{\mathbb{H}^n}u|},\frac{\nabla_{\mathbb{H}^n}u}{|\nabla_{\mathbb{H}^n}u|}\rangle.
\end{split}
\end{equation}

\begin{definition}
Fix a value of $p\in (1,\infty)$ and consider the $p$-Laplace equation 
\begin{equation}\label{plaplace}
-\div_{\mathbb{H}^n}( |\nabla_{\mathbb{H}^n}u|^{p-2} \nabla_{\mathbb{H}^n}u)=0.
\end{equation}
\begin{itemize}
\item[(i)] A lower semi-continuous function $u$ is a viscosity supersolution of  (\ref{plaplace})  if for every $\phi\in C^2(\Omega)$ such that $u-\phi$ has a strict minimum at $P_{0}\in\Omega,$ and $\nabla_{\mathbb{H}^n}\phi(P_{0})\not=0$ we have
$$
-(p-2)\Delta_{\infty, \mathbb{H}^n}\phi(P_{0})-\Delta_{\mathbb{H}^n}\phi(P_{0})\geq 0.
$$
\item[(ii)] A  lower semi-continuous function $u$ is a viscosity subsolution of (\ref{plaplace})  if for every $\phi\in C^2(\Omega)$ such that $u-\phi$ has a strict maximum in $P_{0}\in\Omega,$ and $\nabla_{\mathbb{H}^n}\phi(P_{0})\not=0,$ we have
$$
-(p-2)\Delta_{\infty, \mathbb{H}^n}\phi(P_{0})-\Delta_{\mathbb{H}^n}\phi(P_{0})\leq 0.
$$
\item[(iii)] A continuous function $u$ is a viscosity solution of of  (\ref{plaplace})  if it is both a viscosity supersolution and a viscosity subsolution. 
\end{itemize}
\end{definition}

As shown in \cite{JLM} for the Euclidean case and in \cite{Bi} for the subelliptic  case, it suffices to consider smooth functions whose horizontal gradient does not vanish. 
In addition in those papers it is shown that the notions of viscosity and weak solutions agree for homogeneous equation $-\Delta_{p, \mathbb{H}^n}u=0$.

Next we state carefully what we mean when we say that the asymptotic expansion  (\ref{asymptotic}) holds in the    viscosity sense.  Recall the familiar definition of \lq\lq little  o\rq\rq for a real valued function $h$ defined in a neighborhood of the origin. We write 
$$h(x)=o(x^2)\text{ as } x\to0^+$$ for 
$$\lim_{x\to 0^+} \frac{h(x)}{x^2}=0.$$
\begin{definition} Let $h$ be a real valued function defined in a neighborhood of zero.
We say that
$$h(x)\le o(x^2)\text{ as } x\to0^+$$ if any of the three equivalent conditions is satisfied:
\begin{itemize}
\item[a)] $\displaystyle\limsup_{x\to 0^+} \frac{h(x)}{x^2}\le 0$, 
\item[b)] there exists a nonnegative function $g(x)\ge 0$ such that
$$ h(x)+g(x)=o(x^2) \text{ as } x\to0^+,$$ or
\item[c)] $\displaystyle\lim_{x\to 0^+} \frac{h^+(x)}{x^2}\le 0,$
\end{itemize}
\end{definition}
A similar definition is given for $$h(x) \ge o(x^2)\text{ as }x\to0^+.$$
by reversing the inequalities in a) and c), requiring that $g(x) \le 0$ in b)
and replacing $h^+$ by $h^-$ in c). 
\begin{definition}\label{inequal_viscosity_definition}
A continuous function defined in a neighborhood of a point $P\in \mathbb{H}^{n}$ satisfies
\begin{equation}\label{asymptotic2}
u(P)=\frac{\alpha}{2}\left(\min_{\overline{\BB(P,\epsilon)}}u +\max_{\overline{\BB(P,\epsilon)}}u\right)+  \beta \kint_{\mathbb{\BB}(P,\epsilon)}+o(\epsilon^2),
\end{equation}
as $\epsilon\to 0$
in viscosity sense, if 
\begin{itemize}
\item[(i)] for every continuous function $\phi$ defined in a neighborhood of a point $P$ such that $u-\phi$ has a strict minimum at  $P$  with $u(P)=\phi(P)$ we have
$$
-\phi(P)+\frac{\alpha}{2}\left(\min_{\overline{\BB(P,\epsilon)}}\phi +\max_{\overline{\BB(P,\epsilon)}}\phi\right) +\beta\kint_{\BB(P,\epsilon)}\phi\le o(\epsilon^2),
$$
as $\epsilon \to 0$, and
\item[(ii)] for every continuous function $\phi$ defined in a neighborhood of a point $P$ such that $u-\phi$ has a strict maximum at   $P$  with $u(P)=\phi(P)$ then
$$
\phi(P)-\frac{\alpha}{2}\left(\min_{\overline{\BB(P,\epsilon)}}\phi +\max_{\overline{\BB(P,\epsilon)}}\phi\right)+\beta\kint_{\BB(P,\epsilon)}\phi\ge o(\epsilon^2).
$$
as $\epsilon \to 0$.
\end{itemize}
\end{definition}

For the case $p=\infty$ we could  consider the $1$-homogeneous infinity-Laplacian
\ref{infinitylaplacian}, 
but there remain  some technical difficulties. We certainly conjecture that
the statement of Theorem \ref{maintheorem} holds in this case.

\section{   Heisenberg group Preliminaries}\label{preliminary}
 For $n\ge 1$ we denote by
$\mathbb{H}^n$ the set $\mathbb{R}^{2n+1}$
 endowed with the non-commutative product law given by

$$
(x_1,y_1,t_1)\star(x_2,y_2,t_2)=(x_1+x_2,y_1+y_2,t_1+t_2+2(x_2\cdot y_1-x_1\cdot y_2)),
$$
where we have denote points in $\mathbb{R}^{2n+1}$ as 
$P=(x,y,t)$ with $x,y\in \mathbb{R}^n$ and $t\in\mathbb{R}$ 
and $x\cdot y$ denote the usual inner product in $\mathbb{R}^n$.
The pair  $\mathbb{H}^n\equiv(\mathbb{R}^{2n+1},\star)$ is  the Heisenberg group of order $n$. From now we will denote the group operation $\cdot$ instead of $\star$ when there is no risk of confusion.
\par
Given $P=(x,y,t)\in\mathbb{H}^n$ we write $X_i=(e_i,0,2y_i)$ and $Y_i=(0,e_i,-2x_i)$ for $i=1,\dots,n$,  where $\{e_i\}_{1\leq i\leq n}$ is the canonical basis in $\mathbb{R}^n$.  We identify these vectors with the vector fields 
$$
X_i=\partial_{x_i}+2y_i\partial_t
$$
$$
Y_i=\partial_{y_i}-2x_i\partial_t.
$$
The commutator between the vector fields  $X_i$ and $Y_j$ is $0$ except when the indexes are the same $i=j$, in which case we have
$$
[X_i,Y_i]=-4\partial_t.
$$
\par
The intrinsic (or horizontal)  gradient of a smooth function $u$ at the point $P$ is  given by
$$
\nabla_{\mathbb{H}^n}u(P)=\sum_{i=1}^n(X_iu(P)X_i(P)+Y_iu(P)Y_i(P)).
$$
The horizontal tangent space at the point $P$ is the $2n$-dimensional space 
 $$\mbox{span}\{X_1,\dots,X_n,Y_i,\dots,Y_n\}.$$ 
 We build a metric in the horizontal tangent space by declaring that 
 the set of vectors $\{X_1,\dots,X_n,Y_i,\dots,Y_n\}$ is an orthonormal basis. Thus for every pair of horizontal vectors  $U=\sum_{j=1}^n(\alpha_{1,j}X_{j}(P)+\beta_{1,j}Y_j(P))$  and 
$V=\sum_{j=1}^n(\alpha_{2,j}X_{j}(P)+\beta_{2,j}Y_j(P))$
we have the natural inner product
$$
\langle U,V\rangle=\sum_{i=1}^n\alpha_{1,j}\alpha_{2,j}+\beta_{1,j}\beta_{2,j}.
$$  
In particular we get the corresponding norm 
$$
|U|=\sqrt{\sum_{i=1}^n\left(\alpha_{1,j}^2+\beta_{1,j}^2\right)}.
$$
The norm of the intrinsic gradient of the smooth function $u$ in $P$ is then given by
$$
|\nabla_{\mathbb{H}^n} u(P)|=\sqrt{\sum_{i=1}^n(X_iu(P))^2+(Y_iu(P))^2}
$$

\par

If $\nabla_{\mathbb{H}^n} u(P)=0$ then we say that the point $P$ is characteristic for the
surface $\{u=u(P)\}.$
For every point $P$ that is not characteristic the intrinsic normal to the surface $\{u=u(P)\}$ is  defined,  up to  orientation,  by the unit horizontal normal vector
$$
\nu(P)=\frac{\nabla_{\mathbb{H}^n}u(P)} {|\nabla_{\mathbb{H}^n} u(P)|}\cdot
$$
A semigroup of anisotropic dilations that are  group homomorphism is defined as follows: for every $r>0$ and  
$P\in \mathbb{H}^n$ let
$$
\delta_r(P)=(rx,ry,r^2t).
$$
A smooth gauge that is homogeneous with respect to the anisotropic dilations $\delta_r$ is given by
 $$
\|(x,y,t)\|=\sqrt[4]{(\mid x\mid^2+\mid y\mid^2)^2+t^2}.
$$
We then have $\|\delta_r(P)\|=r \|P\|$. We use this property to define the gauge ball of radius $r$ centered in $0$ iby
$$
B(0,r)=\{P\in \mathbb{H}^n\colon \|P\|<r\}.
$$

The Haar measure in $\mathbb{H}^n$ turns out to be the Lebesgue measure in 
$\mathbb{R}^{2n+1}$. Moreover  we have 
$$
\mid \delta_\lambda (B(0,1))\mid=\lambda^{2n+2}\mid B(0,1)\mid.
$$
For these properties and much more see the book \cite{CDPT}.
\par
The symmetrized horizontal Hessian matrix of the smooth function $u$ at the point $P$ is  the following $2n\times2n$ matrix:
$$
D_{\mathbb{H}^n}^{2*}u(P)=\frac{1}{2} \left(D_{\mathbb{H}^n}^{2}u(P) +
\left(D_{\mathbb{H}^n}^{2} u(P)\right)^t.
\right)
$$
The $(i,j)$-entry of $D_{\mathbb{H}^n}^{2}u$ is given by
\begin{itemize}
\item[(1)]  $X_i u\, X_j u$ for $1\le i,j,\le n$,
\item[(2)]  $X_i u\, Y_{j-n} u$ for $1\le i\le n$, $n+1\le j\le 2n$,
\item[(3)]  $X_{i-n}u\, Y_j u$ for $n+1\le i\le 2n$, $1\le j\le n$, and
\item[(4)]  $X_{i-n}u\, Y_{j-n} u$ for $n+1\le i\le 2n$, $n+1\le j\le 2n$.\
\end{itemize}\par
Next,  we briefly review  the  Taylor Formula adapted to our framework. 
Let $u$ be a smooth function  defined in $\Omega$, an open neighborhood of $0$ .  Let  $\epsilon_0$ be  a positive small number such that  for $\|P\|\le \epsilon_0$ and 
for   $0\leq s\leq 1$  the points $\delta_{s}(P)\in \Omega$. In this way  the function
$$
g(s)=u(\delta_{s}(P))=u(s x,s y,s^2 t)
$$
is well defined for every $s\in[0,1].$ 
By the classical Taylor's formula centered in $0$, we get
$$
g(s)=g(0)+g'(0)s+\frac{1}{2}g''(0)s^2+o(s^2),
$$
as $s\to0^+$.
Computing derivatives we get
\begin{equation}\label{rmk}
g'(s)=\langle\nabla_{\mathbb{H}^n}u(\delta_{s}(P))),(x,y)\rangle+2st\partial_tu(\delta_{s}(P)),
\end{equation}
and
\begin{equation}
\begin{split}
g''(s)
=\langle &D_{\mathbb{H}^n}^{2*}u(\delta_{s}(P))(x,y),(x,y)\rangle\\
+&2st(\partial_tX_iu(\delta_{s}(P))+\partial_tY_iu(\delta_{s}(P))\\
+&2t\partial_tu(\delta_{s}(P))+4s^2t\partial_{tt}u(\delta_{s}(P)).
\end{split}
\end{equation}

Therefore we get the expansion
$$u(\delta_s(P))=u(0)+ \langle\nabla_{\mathbb{H}^n}u(0),(sx,sy)\rangle+2 s^2 t \partial_t u(0)+
\frac{1}{2} \langle D_{\mathbb{H}^n}^{2*}u(0)(sx,sy),(sx,sy)\rangle\\
\\
+ o(s^2).
$$
Writing $Q=\delta_s(P)$, noting that $\|Q\|=s \|P\|$, that
$Q=(sx,sy,s^2t)$ and relabeling we get the horizontal Taylor  expansion valid
 for $P$ near zero
\begin{equation}\label{taylor}
\begin{split}
u(P)=&u(0)+\langle\nabla_{\mathbb{H}^n}u(0),(x,y)\rangle+2\, t\, \partial_tu(0))\\
&+\frac{1}{2}
\left(\langle D_{\mathbb{H}^n}^{2*}u(0)(x,y),(x,y)\rangle+o(\|(x,y,t)\|^2\right).
\end{split}
\end{equation}

\section{Key tools for the proof}\label{Mean_formulas}
\begin{lemma}
Let u be a smooth function. If $\nabla_{\mathbb{H}^n} u(0)\not=0,$ then there exists $\epsilon_0>0$ such that for every $\epsilon\in(0,\epsilon_0)$  there exist points $P_{\epsilon, m},P_{\epsilon, M}\in \partial B(0,\epsilon)$ such that $$\max_{\overline{B(0,\epsilon)}}u=u(P_{\epsilon, M})$$ and $$\min_{\overline{B(0,\epsilon)}}u=u(P_{\epsilon,m}).$$
\end{lemma}
\begin{proof}
Let us consider the case of the maximum, the case of the minimum being analogous. Let us proceed by contradiction. Assume that a sequence of positive numbers $\{\epsilon_j\}_{j\in \mathbb{N}}\subset\mathbb{R}^+$ and a sequence of points $\{P_j\}_{j\in \mathbb{N}}\subset {B(0,\epsilon_j)}$  such that $\epsilon_j\to 0,$ as $j\to+\infty$ and $$\max_{\overline{B(0,\epsilon_j)}}u=u(P_{j}).$$
 Then for every $j\in \mathbb{N},$ we have that $\nabla u(P_j)=0$ because $P_j$ is in the interior of $B(0,\epsilon_j)$. Hence we get a contradiction with the fact that by continuity of $\nabla u$ gives $\nabla u(0)=0$,  which implies $\nabla_{\mathbb{H}^n} u(0)=0.$
\end{proof}
\begin{lemma}
\label{lemma1}
 For small $\epsilon>0$, consider points $P_{\epsilon, M}$ and $P_{\epsilon, m}$ in $\partial B(0,\epsilon)$ such that 
$$\max_{\overline{B(0,\epsilon)}}u=u(P_{\epsilon, M})
\text{  and  } \min_{\overline{B(0,\epsilon)}}u=u(P_{\epsilon, m}).$$
Whenever $\nabla_{\mathbb{H}^n}u(0)\not=0$  we have
$$
\lim_{\epsilon\to 0}\frac{(x_{M,\epsilon},y_{M,\epsilon})}{\epsilon}=\frac{\nabla_{\mathbb{H}^n}u(0)}{\mid \nabla_{\mathbb{H}^n}u(0)\mid} 
$$
and
$$
\lim_{\epsilon\to 0}\frac{(x_{m,\epsilon},y_{m,\epsilon})}{\epsilon}=-\frac{\nabla_{\mathbb{H}^n}u(0)}{\mid \nabla_{\mathbb{H}^n}u(0)\mid},
$$
where $P_{\epsilon}=(x_{\epsilon},y_{\epsilon},t_{\epsilon})\in\mathbb{H}^n$.
Moreover, we also have
$$
\lim_{\epsilon\to 0}\frac{\mid t_\epsilon\mid}{\epsilon^3}=\frac{2\mid u_t (0)\mid}{\mid \nabla_{\mathbb{H}^n}u(0)\mid}
$$
\end{lemma}
\begin{proof}[Proof of Lemma \ref{lemma1}]
We consider the case of the maximum by using the method of Lagrange multipliers. 
There exists $\lambda_\epsilon\in \mathbb{R}$ such that
\begin{equation}
\left\{
\begin{array}{rcl}
u_{x_j}(P_{M,\epsilon})&=&4\lambda_\epsilon\, x_{\epsilon,j}(| x_{\epsilon}|^2+|y_{\epsilon}|^2)\\
u_{y_j}(P_{M,\epsilon})&=&4\lambda_\epsilon \,y_{\epsilon,j}(| x_{\epsilon}|^2+|y_{\epsilon}|^2)\\
u_{t}(P_{M,\epsilon})&=&2\lambda_\epsilon\, t_{\epsilon}\\
(|x_{\epsilon}|2+|y_{\epsilon}|^2)^2+\gamma \, t_{\epsilon}^2&=&\epsilon^4
\end{array}
\right.
\end{equation}
Thus we get
\begin{equation*}
\begin{split}
X_ju(P_{M,\epsilon})&=4\lambda_\epsilon\, x_{\epsilon,j}(|x_{\epsilon}|^2+|y_{\epsilon}|^2)+4\lambda_\epsilon \, y_{\epsilon,j}\,\gamma \,t_{\epsilon}\\
&=4\lambda_\epsilon \left(x_{\epsilon,j}(| x_{\epsilon}|^2+| y_{\epsilon}|^2)+ y_{\epsilon,j} \,t_{\epsilon}\right)
\end{split}
\end{equation*}
and
\begin{equation*}\begin{split}
Y_ju(P_{M,\epsilon})&=4\lambda_\epsilon \,y_{\epsilon,j}(|x_{\epsilon}|^2+|y_{\epsilon}|^2)-4\lambda_\epsilon\,  x_{\epsilon,j}\, \gamma\, t_{\epsilon}\\
&=4\lambda_\epsilon\left(y_{\epsilon,j}(|x_{\epsilon}|^2+|y_{\epsilon}|^2)- x_{\epsilon,j} \,t_{\epsilon}\right)
\end{split}\end{equation*}
To simplify the notation write $\rho_\epsilon=|x_{\epsilon}|^2+|y_{\epsilon}|^2$,  $x_j=x_{\epsilon,j}$, $y_j=y_{\epsilon,j}$ and $\rho_\epsilon=\rho$.
Since the calculation is the same for every $j=1\dots n$,  it suffices to solve the following system
 \begin{equation}
\left\{
\begin{array}{rcl}
Xu&=&4\lambda_\epsilon (\xi\rho^2+ \eta t)\\
Yu&=&4\lambda_\epsilon(\eta \rho^2- \xi t)\\
u_{t}&=&2\lambda_\epsilon t\\
\rho^4+\gamma t^2&=&\epsilon^4,
\end{array}
\right.
\end{equation}
 where $\xi=x_j$ and $\eta=y_j.$
Hence,  it follows that if $u_t(P_{\epsilon,j})\not=0$ we have

\begin{equation}
\left\{
\begin{array}{rcl}
\medskip\medskip
\xi\rho^2+ \eta t&=&\displaystyle{\frac{t \,Xu}{2u_t}}\\
\medskip\medskip
- \xi t+\eta \rho^2&=&\displaystyle{\frac{t\,Yu}{2u_t}}\\
\medskip\medskip
\lambda_\epsilon&=&\displaystyle{\frac{u_{t}}{2t}} \\
\medskip
\rho^4+\gamma t^2&=&\epsilon^4.
\end{array}
\right.
\end{equation}
As a consequence solving for $\xi$ and $\eta$, squaring and adding we get

\begin{equation}
\rho^2=\frac{t^2}{4\epsilon^4}\frac{\mid\nabla_{\mathbb{H}^n}u\mid^2}{u_t^2}.
\end{equation}
Thus,  we obtain:
\begin{equation}\label{eqdelta}
\mid t\mid=2\, \rho\, \epsilon^2\frac{\mid u_t\mid}{\mid\nabla_{\mathbb{H}^n}u\mid}.
\end{equation}
\begin{equation}\label{eqxi2}
\xi=\mbox{sgn}(t) \frac{\rho^3}{\epsilon^2}\frac{Xu}{\mid\nabla_{\mathbb{H}^n}u\mid}- \mbox{sgn}(t)\frac{2\rho^2\mid u_t\mid Yu}{\mid\nabla_{\mathbb{H}^n}u\mid}
\end{equation}
\begin{equation}\label{eqeta2}
\eta=\mbox{sgn}(t)\frac{\rho^3}{\epsilon^2}\frac{Yu}{\mid\nabla_{\mathbb{H}^n}u\mid}+\mbox{sgn}(t)\frac{2\rho^2\mid u_t\mid Xu}{\mid\nabla_{\mathbb{H}^n}u\mid}.
\end{equation}
\noindent Squaring and summing one more time and keeping in mind that $$(\frac{Xu}{\mid\nabla_{\mathbb{H}^n}u\mid})^2+(\frac{Yu}{\mid\nabla_{\mathbb{H}^n}u\mid})^2=1$$
we get 
$$
\rho^2=\frac{\rho^6}{\epsilon^4}+4\rho^4 u_t^2,
$$
which implies
$$
\frac{\rho^4}{\epsilon^4}+4\rho^2 u_t^2=1.
$$
We deduce that  $\rho\sim\epsilon$  whenever $\epsilon\to 0,$ since $u_t(0)$ is bounded. 
As a consequence it follows from (\ref{eqxi2}) and (\ref{eqeta2}) that
$$\frac{\xi}{\epsilon}\to \pm\frac{Xu}{\mid\nabla_{\mathbb{H}^n}u\mid} $$
and
$$\frac{\eta}{\epsilon}\to \pm \frac{Yu}{\mid\nabla_{\mathbb{H}^n}u\mid}. $$

On the other hand recalling (\ref{eqdelta}) and the fact that $\rho\sim\epsilon,$ as $\epsilon\to 0,$ we get 

$$
\lim_{\epsilon\to 0}\frac{|t|}{\epsilon^3}=\frac{2| u_t(0)|}{|\nabla_{\mathbb{H}^n}u(0)|}.
$$
If there is a sequence of point $\{P_{\epsilon_j}\}_{j\in\mathbb{N}}$ such that $u_t(P_{\epsilon_j})=0,$ then either $t_{\epsilon_j}=0,$ or $\lambda_{\epsilon_j}=0.$ In the first case we get that $\rho^4=\epsilon^4,$ so that $\rho=\epsilon.$ Thus we get
$$Xu(P_{\epsilon_j})=4\lambda_{\epsilon_j}\,\xi\epsilon^2$$
and
$$Yu(P_{\epsilon_j})=4\lambda_{\epsilon_j}\,\eta\epsilon^2.$$
As a consequence,  squaring and summing once again, we get
$$|\nabla_{\mathbb{H}^n}u(P_{\epsilon_j})|^2= 16\lambda_{\epsilon_j}^2\epsilon^4\rho^2.$$
 So that  since $|\nabla_{\mathbb{H}^n}u(0)|\not=0,$ this implies 
$$
\lambda_{\epsilon_j}^2\epsilon^4\rho^2\sim \lambda_{\epsilon_j}^2\epsilon^6\to \frac{|\nabla_{\mathbb{H}^n}u(0)|^2}{16},
$$
that is
$$
|\lambda_{\epsilon_j}|\,\epsilon^3\to\frac{|\nabla_{\mathbb{H}^n}u(0)|}{4}.
$$
As a consequence, as $\epsilon_j\to 0$ we get
$$
Xu(P_{\epsilon_j})= 4\lambda_{\epsilon_j}\frac{\xi}{\epsilon}\epsilon^3\sim \pm|\nabla_{\mathbb{H}^n}u(0)|\frac{\xi}{\epsilon}
$$
$$
Yu(P_{\epsilon_j})= 4\lambda_{\epsilon_j}\frac{\eta}{\epsilon}\epsilon^3\sim \pm|\nabla_{\mathbb{H}^n}u(0)|\frac{\eta}{\epsilon},
$$
that is 
$$
\frac{\xi}{\epsilon}\to \pm\frac{Xu(0)}{|\nabla_{\mathbb{H}^n}u(0)|}
$$
and
$$
\frac{\eta}{\epsilon}\to \pm\frac{Yu(0)}{|\nabla_{\mathbb{H}^n}u(0)|}.
$$
In the case that there exists a  sequence of points $\{P_{\epsilon_j}\}_{j\in\mathbb{N}}$ such that  such that $\lambda_{\epsilon_j}=0,$ then we would get
$$
Xu(P_{\epsilon_j})=0 \text{  and  } Yu(P{\epsilon_j})=0,
$$
getting a contradiction with the assumption that $\nabla_{\mathbb{H}^n}u(0)\not=0$.

We just need to justify the sign of the limit. Using the Taylor's formula in the Heisenberg group we get:
\begin{equation*}
\begin{split}
u(P_{\epsilon,M})=& \, u(0)+\langle\nabla_{\mathbb{H}^n}u(0),(x_{\epsilon,M},y_{\epsilon,M})\rangle\\
+&\frac{1}{2}
\left(\langle D_{\mathbb{H}^n}^{2*}u(0)(x_{\epsilon,M},y_{\epsilon,M}),(x_{\epsilon,M},y_{\epsilon,M})\rangle+2t\partial_tu(0))+o(\epsilon^2\right)
\end{split}
\end{equation*}
Hence, dividing by $\epsilon>0$ we get
\begin{equation*}
\begin{split}
0&\leq \frac{u(P_{\epsilon,M})-u(0)}{\epsilon}=\langle\nabla_{\mathbb{H}^n}u(0),(\frac{x_{\epsilon,M}}{\epsilon},\frac{y_{\epsilon,M}}{\epsilon})\rangle\\
&+\frac{1}{2}
\left(\langle D_{\mathbb{H}^n}^{2*}u(0)(\frac{x_{\epsilon,M}}{\epsilon},\frac{y_{\epsilon,M}}{\epsilon}),(x_{\epsilon,M},y_{\epsilon,M})\rangle+2\frac{t_\epsilon}{\epsilon}\partial_tu(0))+o(\epsilon) \right).
\end{split}
\end{equation*}
Letting $\epsilon\to 0$ we get
$$
0\leq \langle\nabla_{\mathbb{H}^n}u(0),\lim_{\epsilon\to 0}(\frac{x_{\epsilon,M}}{\epsilon},\frac{y_{\epsilon,M}}{\epsilon})\rangle, 
$$
which implies 
$$
\lim_{\epsilon\to 0}\left(\frac{x_{\epsilon,M}}{\epsilon},\frac{y_{\epsilon,M}}{\epsilon}\right)=\frac{\nabla_{\mathbb{H}^n}u(0)}{ |\nabla_\mathbb{H}^nu(0)|}.
$$ 
\end{proof}
\begin{lemma}\label{meanformulaH}
Let $u$ be a smooth function defined in a open subset of the Heisenberg group.
Then
$$
\kint_{B(P,\epsilon)}u(x,y,t)=u(P)+C(n)\Delta_{\mathbb{H}^n}u(P)\epsilon^2+o(\epsilon^2),
$$
as $\epsilon\to 0,$ where
$$
C(n)=\frac{1}{2(n+1)}\frac{\int_{0}^{1}(1-s^2)^{\frac{n+1}{2}}\,ds}{\int_{0}^{1}(1-s^2)^{\frac{n}{2}}\,ds}
$$
\end{lemma}
\begin{proof}
Without loss of generality we set $P=0.$ 
We  average the Taylor expansion \ref{taylor} to get 

\begin{equation}
\begin{split}
\kint_{B(0,\epsilon)}u(x,y,t)=&u(0)
+\kint_{ B(0,\epsilon}\langle\nabla_{\mathbb{H}^n}u(0),(x,y)\rangle\\&+
\frac{1}{2}\int_{ B(0,\epsilon}
\langle D_{\mathbb{H}^n}^{2*}u(0)(x,y),(x,y)\rangle+
\kint_{ B(0,\epsilon}2t\partial_tu(0))+o(\|(x,y,t)\|^2\\
&=u(0)+\frac{1}{2}\int_{B(0,\epsilon)}\langle D_{\mathbb{H}^n}^{2*}u(0)(x,y),(x,y)\rangle +\int_{B(0,\epsilon)}o(\|(x,y,t)\|^2\\
&=u(0)+C(n)\Delta_{\mathbb{H}^n}u(P)\epsilon^2+o(\epsilon^2).
\end{split}
\end{equation}
Indeed we have that the linear terms vanish
\begin{equation*}
\int_{ B(0,\epsilon)}\langle\nabla_{\mathbb{H}^n}u(0),(x,y)\rangle\, dxdydt=\int_{-\epsilon^2}^{\epsilon^2}\int_{S(\epsilon, t)}\langle\nabla_{\mathbb{H}^n}u(0),(x,y)\rangle dxdydt\\
=0,
\end{equation*}
where we have set 
$S(\epsilon, t)= \{(x,y)\colon | x|^2+|y|^2<\sqrt[4]{\epsilon^4-t^2}\}.$
Proceeding analogously for the second order terms we get 
\begin{equation}
\begin{split}
\int_{ B(0,\epsilon)}\langle D_{\mathbb{H}^n}^{2*}u(0)(x,y),&(x,y)\rangle \, dxdydt=\\
&=\int_{-\epsilon^2}^{\epsilon^2}\int_{S(\epsilon, t)}\langle D_{\mathbb{H}^n}^{2*}u(0)(x,y),(x,y)\rangle \, dxdydt\\
&=\sum_{i=1}\left(X_i^2u(0)+Y_i^2u(0)\right)\int_{-\epsilon^2}^{\epsilon^2}\int_{S(\epsilon, t)}(x_i^2+y_i^2) \, dxdydt\\
&=\frac{1}{(n+1)}\frac{\int_{0}^{1}(1-s^2)^{\frac{n+1}{2}}\,ds}{\int_{0}^{1}(1-s^2)^{\frac{n}{2}}\,ds}\,|B(0,\epsilon)|\,\Delta_{\mathbb{H}^n}u(0)\,\epsilon^{2}
\end{split}
\end{equation}
since 
\begin{equation}
\begin{split}
\int_{-\epsilon^2}^{\epsilon^2}\int_{S(\epsilon, t)}(x_i^2+y_i^2) \, dxdydt=&\frac{1}{n}\int_{-\epsilon^2}^{\epsilon^2}\int_{S(x,t)}(| x|^2+| y|^2)\,  dxdydt\\
=&\frac{\omega(2n)}{n}\int_{-\epsilon^2}^{\epsilon^2}\int_0^{\sqrt[4]{\epsilon^4-t^2}}\rho^{2}\rho^{2n-1}d\rho dt\\
=&\frac{\omega(2n)}{n(n+1)}\int_{0}^{\epsilon^2}(\epsilon^4-t^2)^{\frac{n+1}{2}}dt\\
=&\frac{\omega(2n)}{n(n+1)}\left(\int_{0}^{1}(1-s^2)^{\frac{n+1}{2}}dt\right)\epsilon^{2n+4}\\
=&\frac{1}{(n+1)}\frac{\int_{0}^{1}(1-s^2)^{\frac{n+1}{2}}\,dt}{\int_{0}^{1}(1-s^2)^{\frac{n}{2}}\,dt}|B(0,\epsilon)|\, \epsilon^{2}
\end{split}
\end{equation}
where we have denoted by $\omega(n)$ the Euclidean surface area of they unit sphere $\partial B(0,1)$ in $\mathbb{R}^{2n}$,  
and have used the formula $|B(0,\epsilon)|=\epsilon^{2n+2}\, \frac{\omega(2n)}{n}\int_{0}^{1}(1-t^2)^{\frac{n}{2}}\,dt$ for
the Lebesgue measure in $\mathbb{R}^{2n+1}$ of the gauge ball of radious 
$\epsilon$. 
\end{proof}
\section{Proof of Theorem \ref{maintheorem}}
The first step in the proof of Theorem \ref{maintheorem} is the  following expansion valid for smooth functions.

Let  $P\in\Omega$ be a point and  $\phi$ be a $C^2$-function defined in a neighborhood of $P.$ 
We denote by $P_{\epsilon, M}\in \overline{B(P,\epsilon)}$ and $P_{\epsilon,m}\in \overline{B(P,\epsilon)}$ the points of maxima and minima
$$\phi(P_{\epsilon, M})=\max_{\overline{B(P,\epsilon)}}\phi\\,\,\text{    and   }\,\,\phi(P_{\epsilon, m})=\min_{\overline{B(P,\epsilon)}}\phi.$$

\begin{lemma}\label{lemmakey}
Let $p\in (1, +\infty)$and $\phi$ be a $C^2$-function in a domain $\Omega\subset\mathbb{H}^n$. Let $C(n)$, $\alpha$, $\beta$  be given in the statement of Theorem \ref{maintheorem}. 
Consider the vectors
	$$(h_{\epsilon}, 
	l_{\epsilon})=\left(\frac{x_{\epsilon, M}-x}{\epsilon},
	\frac{y_{\epsilon, M}-y}{\epsilon}\right).
	$$
\noindent The following expansions hold near every $P\in\Omega$, 

\begin{equation*}
\begin{split}
 \beta\,  C(n)\epsilon^2 
 \bigg[
 \Delta_{\mathbb{H}^n}\phi(P)+&(p-2)
 \langle
  D_{\mathbb{H}^n}^{2,*}
 \phi(P)(h_{\epsilon}, l_{\epsilon}),(h_{\epsilon}, l_{\epsilon})
 \rangle
 \bigg]
 \ge \\
& \beta\kint_{B(P,\epsilon)}\phi(x,y,t)+\frac{\alpha}{2}\left(\min_{\overline{B(P,\epsilon)}}\phi +\max_{\overline{B(P,\epsilon)}}\phi\right) -\phi(P)+ o(\epsilon^{2}),
\end{split}
\end{equation*}
as $\epsilon\to 0$
and
\begin{equation*}
\begin{split}
 \beta\,  C(n)\epsilon^2 
 \bigg[
 \Delta_{\mathbb{H}^n}\phi(P)+&(p-2)
 \langle
  D_{\mathbb{H}^n}^{2,*}
 \phi(P)(h_{\epsilon}, l_{\epsilon}),(h_{\epsilon}, l_{\epsilon})
 \rangle
 \bigg]
 \le \\
& \beta\kint_{B(P,\epsilon)}\phi(x,y,t)+\frac{\alpha}{2}\left(\min_{\overline{B(P,\epsilon)}}\phi +\max_{\overline{B(P,\epsilon)}}\phi\right) -\phi(P)+ o(\epsilon^{2}),
\end{split}
\end{equation*}$\epsilon\to 0.$
\end{lemma}
\begin{proof}
We can assume without any restriction that $P=0$ ($x=0$, $y=0$, and $t=0$) just moving $P$ to the origin by a left translation of the group.The  Taylor formula in the Heisenberg group gives
\begin{equation*}
\begin{split}
\phi(P_{\epsilon, M})=\phi(0)+&\langle\nabla_{\mathbb{H}^n}\phi(0),(x_{\epsilon, M},y_{\epsilon, M})\rangle+\frac{1}{2}
\langle D_{\mathbb{H}^n}^{2,*}\phi(0)(x_{\epsilon, M},y_{\epsilon, M}),
(x_{\epsilon, M},y_{\epsilon, M})\rangle \\
+& 2 t_{\epsilon, M}\,\partial_t\phi(0)+o(\epsilon^2).
\end{split}
\end{equation*}
and
\begin{equation*}
\begin{split}
\phi(-P_{\epsilon, M})=\phi(0)-&\langle\nabla_{\mathbb{H}^n}\phi(0),(x_{\epsilon, M},y_{\epsilon, M})\rangle+\frac{1}{2}
\langle D_{\mathbb{H}^n}^{2,*}\phi(0)(x_{\epsilon, M},y_{\epsilon, M}),
(x_{\epsilon, M},y_{\epsilon, M})\rangle \\
-& 2 t_{\epsilon, M}\,\partial_t\phi(0)+o(\epsilon^2).
\end{split}
\end{equation*}
Adding the last two inequalities we get
\begin{equation*}
\begin{split}
\phi(P_{\epsilon,M})+\phi(-P_{\epsilon,M})=&2\phi(0)+\langle D_{\mathbb{H}^n}^{2*}\phi(0)(x_{\epsilon, M},y_{\epsilon, M}),(x_{\epsilon, M},y_{\epsilon, M})\rangle+o(\epsilon^2).\\
 \end{split}
\end{equation*}
Using the definition of $P_{\epsilon,M}$ if follows that

\begin{equation}
\begin{split}
\max_{\overline{B(0,\epsilon)}}\phi+\min_{\overline{B(0,\epsilon)}}\phi &\le 
\max_{\overline{B(0,\epsilon)}}\phi+\phi(-P_{\epsilon, M})\\
&=\phi(0)+\langle D_{\mathbb{H}^n}^{2*}\phi(0)(x_{\epsilon, M},y_{\epsilon, M}),(x_{\epsilon, M},y_{\epsilon, M})\rangle+o(\epsilon^2), 
\end{split}
\end{equation}
which implies the inequality
\begin{equation}
\begin{split}
\phi(0)+\frac{1}{2}\langle D_{\mathbb{H}^n}^{2*}\phi(0)(x_{\epsilon, M},y_{\epsilon, M}),(x_{\epsilon, M},y_{\epsilon, M})\rangle\geq \frac{1}{2}\left(\max_{\overline{B(0,\epsilon)}}\phi+\min_{\overline{B(0,\epsilon)}}\phi\right)+o(\epsilon^2).\
\end{split}
\end{equation}
Multiplying this relation by $\alpha$, the expansion in Lemma \ref{meanformulaH} by
$\beta$, adding and using the fact that $\alpha+\beta=1$ we obtain
\begin{equation*}
\begin{split}
 \phi(0)+&C(n)\,\beta\,\Delta_{\mathbb{H}^n}\phi(0)\epsilon^2+\frac{\alpha}{2}\langle D_{\mathbb{H}^n}^{2*}\phi(0)(x_{\epsilon, M},y_{\epsilon, M}),(x_{\epsilon, M},y_{\epsilon, M})\rangle\\
&\geq \beta\kint_{B(0,\epsilon)}\phi(x,y,t)+\frac{\alpha}{2}\left(\min_{\overline{B(0,\epsilon)}}\phi +\max_{\overline{B(0,\epsilon)}}\phi\right)+o(\epsilon^2).
\end{split}
\end{equation*}
as we wanted to show in the case $\alpha>0$. 

We determine  $\alpha$ and $\beta$ in such a way that
$$
\frac{\alpha}{2C(n)\beta}=p-2,
$$ 
Thus together the requirement  $\alpha+\beta=1$ we get
$$
\frac{1-\beta}{2C(n)\beta}=p-2,  $$
giving 
$$
\alpha=\frac{2(p-2)C(n)}{2(p-2)C(n)+1}\text{ and } 
 \beta=\frac{1}{2(p-2)C(n)+1}.
$$ 
We can now write
\begin{equation*}
\begin{split}
&  \frac{\epsilon^2C(n)}{2(p-2)C(n)+1}\left(\Delta_{\mathbb{H}^n}\phi(0)+(p-2)\langle D_{\mathbb{H}^n}^{2*}\phi(0)(\frac{x_{\epsilon, M}}{\epsilon},\frac{y_{\epsilon, M}}{\epsilon}),(\frac{x_{\epsilon, M}}{\epsilon},\frac{y_{\epsilon, M}}{\epsilon})\rangle\right) +o(\epsilon^2)\\
&\geq \left(\frac{1}{2(p-2)C(n)+1}\kint_{B(0,\epsilon)}\phi+\frac{(p-2)C(n)}{2(p-2)C(n)+1}\left[\min_{\overline{B(0,\epsilon)}}\phi +\max_{\overline{B(0,\epsilon)}}\phi\right] -\phi(0)\right).
\end{split}
\end{equation*}
This computation works for $\alpha\geq 0$;  that is for every $p\geq 2$ 
When $\alpha<0$ the procedure is the same but the sign of the inequality is reversed, that is
\begin{equation*}
\begin{split}
&  \frac{\epsilon^2C(n)}{2(p-2)C(n)+1}\left(\Delta_{\mathbb{H}^n}\phi(0)+(p-2)\langle D_{\mathbb{H}^n}^{2*}\phi(0)(\frac{x_{\epsilon, M}}{\epsilon},\frac{y_{\epsilon, M}}{\epsilon}),(\frac{x_{\epsilon, M}}{\epsilon},\frac{y_{\epsilon, M}}{\epsilon})\rangle\right) +o(\epsilon^2)\\
&\leq \left(\frac{1}{2(p-2)C(n)+1}\kint_{B(0,\epsilon)}\phi+\frac{(p-2)C(n)}{2(p-2)C(n)+1}\left[\min_{\overline{B(0,\epsilon)}}\phi +\max_{\overline{B(0,\epsilon)}}\phi\right] -\phi(0)\right).
\end{split}
\end{equation*}
and $p\in (1,2)$.
Arguing with the inequality coming from the minimum 
 we get

\begin{equation*}
\begin{split}
& \phi(0)+\frac{1}{2}\langle D_{\mathbb{H}^n}^{2*}\phi(0)(x_{\epsilon, m},y_{\epsilon, m}),(x_{\epsilon, m},y_{\epsilon, m})\rangle+o(\epsilon^2)\leq \frac{1}{2}\left(\min_{\overline{B(0,\epsilon)}}\phi +\max_{\overline{B(0,\epsilon)}}\phi\right).
\end{split}
\end{equation*}
and 
\begin{equation*}
\begin{split}
&  \frac{\epsilon^2C(n)}{2(p-2)C(n)+1}\left(\Delta_{\mathbb{H}^n}\phi(0)+(p-2)\langle D_{\mathbb{H}^n}^{2*}\phi(0)(x_{\epsilon, m},y_{\epsilon, m}),(x_{\epsilon, m},y_{\epsilon, m})\rangle\right) +o(\epsilon^2)\\
&\leq \left(\frac{1}{2(p-2)C(n)+1}\kint_{B(0,\epsilon)}\phi+\frac{(p-2)C(n)}{2(p-2)C(n)+1}\left[\min_{\overline{B(0,\epsilon)}}\phi +\max_{\overline{B(0,\epsilon)}}\phi\right] -\phi(0)\right),
\end{split}
\end{equation*}
for $p\geq 2$ 
and
\begin{equation}\label{recalling1}
\begin{split}
&  \frac{\epsilon^2C(n)}{2(p-2)C(n)+1}\left(\Delta_{\mathbb{H}^n}\phi(0)+(p-2)\langle D_{\mathbb{H}^n}^{2*}\phi(0)\frac{(x_{\epsilon, m},y_{\epsilon, m})}{\epsilon},\frac{(x_{\epsilon, m},y_{\epsilon, m})}{\epsilon}\rangle\right) +o(\epsilon^2)\\
&\geq \left(\frac{1}{2(p-2)C(n)+1}\kint_{B(0,\epsilon)}\phi+\frac{(p-2)C(n)}{2(p-2)C(n)+1}\left[\min_{\overline{B(0,\epsilon)}}\phi +\max_{\overline{B(0,\epsilon)}}\phi\right] -\phi(0)\right),
\end{split}
\end{equation}
for $p\in (1,2)$. 
 
\end{proof}

\begin{proof}[Proof of Theorem \ref{maintheorem}. ]

Suppose that $u$ satisfies the asymptotic expansion in the viscosity sense as in Definition \ref{inequal_viscosity_definition}. Let $\phi$ be a smooth function such that $u-\phi$ has a strict maximum at $P$ and $\nabla_{\mathbb{H}^n}\phi(P)\not =0$. Then it follows, by condition (ii) in Definition \ref{inequal_viscosity_definition},

\begin{equation*} 
\begin{split}
\frac{1}{2(p-2)C(n)+1}\kint_{B(P,\epsilon)}\phi+\frac{(p-2)C(n)}{2(p-2)C(n)+1}\left[\min_{\overline{B(P,\epsilon)}}\phi +\max_{\overline{B(P,\epsilon)}}\phi\right] -\phi(P)\geq 0,
\end{split}
\end{equation*}
and recalling \ref{recalling1}
we conclude that 
$$
\frac{\epsilon^2C(n)}{2(p-2)C(n)+1}\left(\Delta_{\mathbb{H}^n}\phi(P)+(p-2)\langle D_{\mathbb{H}^n}^{2*}\phi(P)\frac{(x_{\epsilon, m},y_{\epsilon, m})}{\epsilon},\frac{(x_{\epsilon, m},y_{\epsilon, m})}{\epsilon}\rangle\right) \ge o(\epsilon^2).
$$
Dividing by $\epsilon^{2}$, using  Lemma \ref{lemma1} and letting $\epsilon\to 0$ we get 

\begin{equation*}
\begin{split}
  \Delta_{\mathbb{H}^n}\phi(P)+(p-2)\langle D_{\mathbb{H}^n}^{2*}\phi(P)\frac{\nabla_{\mathbb{H}^n}\phi(P)}{\mid \nabla_{\mathbb{H}^n}\phi(P)\mid},\frac{\nabla_{\mathbb{H}^n}\phi(P)}{\mid \nabla_{\mathbb{H}^n}\phi(P)\mid}\rangle\geq 0,
\end{split}
\end{equation*}
that is $u$ is a viscosity subsolution of $\Delta_{\mathbb{H}^n,p}u=0.$

Let us prove the converse implication. Assume that $u$ is a viscosity solution. In particular $u$ is a supersolution so that  for every $C^2$ test function $\phi$ such that $u-\phi$ is a strict minimum at the point $P\in \Omega$ with $\nabla_{\mathbb{H}^n}\phi(P)\not=0$ we  have
$$
-(p-2)\Delta_{\mathbb{H}^n\infty}\phi(P)-\Delta_{\mathbb{H}^n}\phi(P)\geq 0.
$$
Recalling  inequality (\ref{recalling1})
\begin{equation*} 
\begin{split}
& 0\geq \frac{\epsilon^2C(n)}{2(p-2)C(n)+1}\left(\Delta_{\mathbb{H}^n}\phi(P)+(p-2)\langle D_{\mathbb{H}^n}^{2*}\phi(P)\frac{(x_{\epsilon, m},y_{\epsilon, m})}{\epsilon},\frac{(x_{\epsilon, m},y_{\epsilon, m})}{\epsilon}\rangle\right) \\
&\geq \left(\frac{1}{2(p-2)C(n)+1}\kint_{B(0,\epsilon)}\phi+\frac{(p-2)C(n)}{2(p-2)C(n)+1}\left[\min_{\overline{B(P,\epsilon)}}\phi +\max_{\overline{B(P,\epsilon)}}\phi\right] -\phi(P)\right)+o(\epsilon^2),
\end{split}
\end{equation*}
and keeping in mind that $$
\lim_{\epsilon\to 0}(\frac{x_{\epsilon, m}}{\epsilon},\frac{y_{\epsilon, m}}{\epsilon})=-\frac{\nabla_{\mathbb{H}^n}\phi(P)}{\mid \nabla_{\mathbb{H}^n}\phi(P)\mid},
$$
we get
$$
\frac{1}{2(p-2)C(n)+1}\kint_{B(P,\epsilon)}\phi\,+\frac{(p-2)C(n)}{2(p-2)C(n)+1}\left[\min_{\overline{B(P,\epsilon)}}\phi +\max_{\overline{B(P,\epsilon)}}\phi\right] -\phi(P)+o(\epsilon^2)\leq 0, 
$$
which is condition (i) in the Definition \ref{inequal_viscosity_definition}. An analogous computation gives the proof of condition (ii).

\end{proof}

\end{document}